\documentclass{amsart}

\usepackage{amssymb}
\usepackage{mathrsfs}
\usepackage{cite}
\usepackage{graphicx}
\usepackage{caption}
\usepackage{subcaption}
\usepackage{float}

\newtheorem{thm}{Theorem}[section]

\newtheorem{lem}[thm]{Lemma}
\newtheorem{cor}[thm]{Corollary}

\theoremstyle{definition}

\theoremstyle{remark}
\newtheorem{remark}[thm]{Remark}

\numberwithin{equation}{section} %numbers equations within sections

\DeclareMathOperator{\N}{\mathbb{N}}  %natural numbers
\DeclareMathOperator{\e}{e}                  % exponential function
\DeclareMathOperator{\dtau}{d\it{\tau}}  %integral
\DeclareMathOperator{\ds}{d\it{s}}          %integral
\DeclareMathOperator{\dt}{d\it{t}}          %integral
\DeclareMathOperator{\dW}{\mathrm{d_W}} %weak distance
\DeclareMathOperator{\dS}{\mathrm{d_S}} %strong distance
\DeclareMathOperator{\g}{\gamma}
\DeclareMathOperator{\la}{\lambda}

\DeclareMathOperator{\ddt}{\frac{d}{d\textit{t}}} %derivative operator w.r.t. t
\DeclareMathOperator{\ddtau}{\frac{d}{d\textit{$\tau$}}}

\begin{document}

\title[Regularizing Effect in the Dyadic Model]{Regularizing Effect of the Forward Energy Cascade in the Inviscid Dyadic Model}
\author{Alexey Cheskidov and Karen Zaya}
\address{Department of Mathematics, Statistics, and Mathematical Computer Science. University of Illinois at Chicago. 851 South Morgan Street (M/C 249). Chicago IL, 60607}
\email{acheskid@uic.edu, kzaya2@uic.edu}

\begin{abstract}
We study the inviscid dyadic model of the Euler equations and prove some regularizing properties of the nonlinear term that occur due to forward energy cascade. We show every solution must have $\frac{3}{5}$ $L^2$-based (or $\frac{1}{10}$ $L^3$-based) regularity for all positive time. We conjecture this holds up to Onsager's scaling, where the $L^2$-based exponent is $\frac{5}{6}$ and the $L^3$-based exponent is $\frac{1}{3}$.
\end{abstract}

 \maketitle

%\tableofcontents

%% A blank line starts a new paragraph
 
\section{Introduction} 
Consider the Navier-Stokes equations 
\begin{align} \label{eq:NSE}
&\frac{\partial u}{\partial t} + (u \cdot \nabla) u = -\nabla p  + \nu \Delta u,\\
&\nabla \cdot u = 0,
\end{align}
where $u$ is the velocity vector field, $p$ is pressure, and $\nu  > 0$ is the viscosity coefficient. Regularity of the three-dimensional Euler equations (when $\nu = 0$) and the Navier-Stokes equations continue to be compelling questions. In this paper we examine the inviscid dyadic model, which shares important characteristics with the 3D Euler equations, namely formal conservation of energy and the scaling properties of the nonlinear term. The role of the nonlinear term is pivotal in the study of turbulent flows. The basic principle proposed by Kolmogorov \cite{MR1124922} behind turbulence is forward energy cascade.  Simply put, the theory asserts that energy moves from large to small scales until it reaches the dissipation range, where the viscous forces dominate. For the Navier-Stokes equations, the dissipation range is the only tool used to prove regularity of solutions, but the forward energy cascade might also be a mechanism to regularize solutions. For quasilinear scalar equations, the regularizing property of the nonlinear term has been studied by Tadmor and Tao in \cite{MR2342955}, but such results remain out of reach for the Euler or Navier-Stokes equations. Very recently, Tao \cite{TaoBlowup} proved blow-up for averaged Navier-Stokes equations by reducing the equations to a more complicated dyadic model where he introduced a delay in energy cascade. This delay seems to destroy the regularizing effect of the nonlinear term studied in this paper and produces a strong blow-up.

Shell models are designed to capture energy cascade in turbulent fluid flows. The dyadic model is a specific example where the nonlinearity is simplified to reflect just the local interactions between neighboring scales. It was initially introduced in 1974 by Desnianskii and Novikov \cite{Desnianskii1974468} in the context of oceanography. Other derivations have been since developed and we refer the reader to \cite{MR2522972} for a more detailed explanation via Littlewood-Paley decomposition. Further mathematical analysis has led to several other results in the last decade, see for example \cite{MR2746670}, \cite{MR2844828}, \cite{MR2415066}, \cite{MR2038114}, \cite{MR2095627}, and \cite{MR2180809}.

The inviscid dyadic model is an infinite system of nonlinearly coupled ordinary differential equations constructed to mimic the behavior of the energy of solutions to the Euler equations in dyadic shells. In \cite{MR2422377}, Cheskidov, Constantin, Friedlander, and Shvydkoy examined the energy flux $\Pi_j$ due to the nonlinearity in the Euler equations through the shell of radius $\la_j=2^j$ and obtained the bound
\[
|\Pi_j| \lesssim \sum_{i=-1}^\infty \la^{-\frac{2}{3}}_{|j-i|}\la_i \| \Delta_i u \|_3^3,
\]
where $\Delta_i u$ is a Littlewood-Paley piece of $u$. Recall Bernstein's inequality:
\[
\|\Delta_j u\|_q \lesssim \la_j^{3(\frac{1}{p}-\frac{1}{q})} \|\Delta_j u\|_p, \hspace{5mm} 1\le p \le q .
\]
We assume $\|\Delta_j u\|_3 \sim \la_j^\beta \|\Delta_j u\|_2$ where $\beta \in [0, \frac{1}{2}] $ is the intermittency parameter. Kolmogorov's regime corresponds to $\beta = 0$, whereas $\beta = \frac{1}{2}$ gives extreme intermittency. Denote the total energy in the $j^{th}$ shell by $a_j^2(t)$. As in \cite{MR2522972}, assuming only local interactions and extreme intermittency, we model the flux through the $j^{th}$ shell of radius $\la_j$ as $\Pi_j = \la_j^\frac{5}{2} a_j^2 a_{j+1}$. This leads to the following inviscid system
\begin{align} \label{eq:aj}
\begin{split}
&\ddt a_j(t) = \la^\frac{5}{2}_{j-1} a_{j-1}^2(t) -\la^\frac{5}{2}_j a_j (t)a_{j+1}(t), \hspace{1cm} j = 1, 2, ... \\
&a_0(t) = 0,
\end{split}
\end{align}
with initial conditions $a_j(0) = a_j^0$ for $j= 0, 1, ...$. 	

Kolmogorov predicted energy cascade produces dissipation anomaly, which is characterized by the persistence of non-vanishing energy dissipation in the limit of vanishing viscosity. This phenomenon is possibly related to anomalous dissipation, failure of the energy to be conserved despite the absence of viscosity. Onsager conjectured that sufficiently rough solutions to Euler's equation can exhibit anomalous dissipation, however if the solution is smooth enough, then the energy should be conserved \cite{MR0036116}. Anomalous dissipation and loss of regularity a priori seem unconnected, but a more discernible relationship exists in the context of the inviscid dyadic model and Onsager scaling. 
%The regularity of solutions is related to the natural scaling of the equations and for the dyadic model we suspect that the natural space for regularity is the Onsager space--not the critical space. Despite the absence of viscosity, in the inviscid dyadic model, a solution with rough initial data immediately gains regularity. This is due to the forward energy cascade and the smoothing properties of the nonlinear term.

In \cite{MR2337019} and \cite{MR2600714}, Cheskidov, Friedlander, and Pavlovi\'{c} showed that all the solutions of the forced inviscid dyadic model must have Onsager's regularity almost everywhere in time and confirmed anomalous dissipation and dissipation anomaly. They also showed that all solutions blow-up in finite time in $H^\frac{5}{6}$. On the other hand, all solutions are in $H^\theta$ for almost all time for $\theta <\frac{5}{6}$. In \cite{MR3057168}, Barbato and Morandin studied the unforced inviscid model and showed Onsager regularity almost everywhere, as well. In addition, they demonstrated that solutions remain in $H^{\frac{1}{2}-}$ for all time.  We improve their result by showing that regularity even closer to Onsager's is retained (see comments after Theorem \ref{thm:reg}). It is natural to conjecture that every solution must have exactly Onsager's regularity for all positive time.  

The main results are in Section 4, where we show
\begin{thm} \label{thm:reg}
For any positive solution to \eqref{eq:aj} with initial condition $a(0)$ in $l^2$, 
\begin{align} \label{eq:sup}
\sup_{j} \la_j^\theta a_j(t) < \infty
\end{align} 
for $t>0$ and $\theta = \frac{3}{5}$.
\end{thm}
Barbato and Morandin proved the theorem for $\theta = \frac{1}{2}$ by finding an 
invariant region for solutions. The method presented in this paper is different as we use a 
more dynamical approach which allows us to improve regularity  for values of $\theta$ up 
to $\frac{3}{5}$. The ultimate goal would be to show regularity for values of $\theta$ up to 
$\frac{5}{6}$, which corresponds to Onsager's scaling. 

\begin{remark}
As a comparison to $L^3$-based regularity, our result \eqref{eq:sup} can be written as
\[
\sup_j \la_j^{q+\beta} a_j(t) < \infty
\]
for $q= \frac{1}{10}$. The ultimate Onsager scaling is $q = \frac{1}{3}$.
\end{remark}

%%%%%%%%%%%%%%%%%%%%%%%%%%%%%%%%%%%%%%%%%%%%%%%
\section{Energy Conservation and Onsager's Conjecture}
Denote $H=l^2$ and define the scalar product and norm (called the energy norm) in the usual manner
\[
(a, b) := \sum_{j=0}^\infty a_j b_j, \hspace{2cm} |a| := \left( \sum_{j=0}^\infty a_j^2 \right)^\frac{1}{2}.
\]
A solution $a(t)$ is called positive if $a_j(t) \ge 0$ for all $j \in \N$ and all time $t$. 

\begin{thm} \label{thm:positive}
Let $a(t)$ be a solution to \eqref{eq:aj} such that $a_j(0) \ge 0$ for all $j \in \N$. Then $a_j(t) \ge 0 $ for all $j \in \N$ and all $t>0$.
\end{thm}
See \cite{MR2600714} and \cite{MR3057168}. Moreover in \cite{MR3057168}, Barbato and Morandin proved uniqueness for positive initial data.

In this section we illustrate why $\theta = \frac{5}{6}$ corresponds to Onsager's scaling by proving the following theorem (cf. \cite{MR2422377}, \cite{MR2665030}):
\begin{thm} \label{thm:energy}
Let $a(t)$ be a positive solution to \eqref{eq:aj} such that 
\[
\lim_{j \rightarrow \infty} \int_0^T \left( \la_j^\frac{5}{6}a_j(t) \right)^3 ~\dt =0
\]
  then $a(t)$ conserves energy on $[0, T]$.
\end{thm}
\begin{proof}
We examine the total energy flux through the first $J^{th}$ shells. To do this, we multiply equations \eqref{eq:aj} by $a_j(t)$, take the finite sum from $j=0$ to $j=J$, and integrate over time  for $0\le t \le T$ to obtain
\[
\int_0^t \sum_{j=0}^J a_j a_j' ~\dtau = \int_0^t \sum_{j=0}^J \left( \la_{j-1}^\frac{5}{2} a_{j-1}^2 a_j- \la_j^\frac{5}{2} a_j^2 a_{j+1} \right) ~\dtau.
\]
The righthand sum telescopes and we rewrite the left side
\[
\int_0^t \sum_{j=0}^J \frac{1}{2} \ddtau ( a_j^2)~\dtau = -\int_0^t \la_J^\frac{5}{2}a_J^2 a_{J+1} \dtau,
\]
which yields
\begin{align} \label{eq:energy}
\frac{1}{2}  \sum_{j=0}^J a_j^2(t) - \frac{1}{2} \sum_{j=0}^J a_j^2(0) = -\int_0^t \la_J^\frac{5}{2}a_J^2(\tau) a_{J+1}(\tau)~\dtau.
\end{align}
Now consider the integral on the righthand side. By Young's inequality, we have
\begin{align*}
0 \le \int_0^t \la_J^\frac{5}{2}a_J^2 a_{J+1} ~\dtau 
&\le \int_0^t \la_J^\frac{5}{2} \left(  \frac{(a_J^2)^\frac{3}{2}}{3/2} + \frac{(a_{J+1})^3}{3}      \right) ~\dtau \\
&\le \int_0^t \la_J^\frac{5}{2}a_J^3 ~\dtau + \int_0^t \la_{J+1}^\frac{5}{2} a_{J+1}^3 ~\dtau.
\end{align*}
Hence by our assumption, 
\[
\lim_{J \rightarrow \infty} \int_0^t \la_J^\frac{5}{2}a_J^2 a_{J+1} ~\dtau = 0.
\]
We take the limit of \eqref{eq:energy} as $J$ goes to infinity to conclude that energy is conserved since $|a(t)|^2 = |a(0)|^2$. 

\end{proof}

%%%%%%%%%%%%%%%%%%%%%%%%%%%%%%%%%%%%%%%%%%%%%%%
\section{The Modified Galerkin Approximation with Flux} 
Define the strong and weak distances, denoted respectively by $\dS$ and $\dW$, as: 
\[
\dS(a, b):= |a-b|, \hspace{2cm} \dW(a,b) := \sum_{j = 0} ^\infty 
\frac{1}{\la^{(j)^2}} \frac{\left| a_j -  b_j \right|}{ 1+ \left| a_j - b_j\right|}.
\]
Also define the modified Galerkin approximation with flux, denoted by
$$
a^n(t) = (a_0^n(t), a_1^n(t), ..., a_n^n(t), 0, ...),
$$ 
to be a solution to the following finite system of ordinary differential equations: 
\begin{align} \label{eq:galerkin}
\begin{split}
&\ddt a_j^n  - \la^\frac{5}{2}_{j-1} (a_{j-1}^n)^2 +\la^\frac{5}{2}_j  a_j^n a_{j+1}^n = 0, \hspace{1cm} j=1, 2, ..., n-1, \\
&\ddt a_n^n  -\la^\frac{5}{2}_{n-1}(a_{n-1}^n)^2 + \la^{\frac{5}{2}-2\theta} \la^{\frac{5}{2}-\theta}_n   a_n^n = 0,
\end{split}
\end{align}
with $a_j^n(0)  = a_j^0$ for $j = 1, 2, ..., n$, where $\theta$ is any positive number. 

By a similar argument to Theorem 3.2 from \cite{MR2600714}, we obtain the following theorem:

\begin{thm} \label{thm:galerkin}
The sequence of the modified Galerkin approximation with flux converges to a solution of the dyadic model \eqref{eq:aj}. 
\end{thm}
\begin{proof}
Denote $a(0)=a^0$, such that $a^0 \in H$ and let $T>0$ be arbitrary. We will show that the modified Galerkin approximation with flux converges to a solution of \eqref{eq:galerkin} on $[0, T]$. We know there exists a unique solution $a^n(t)$ to \eqref{eq:galerkin} from the theory of ordinary differential equations.  We will show the system of Galerkin approximations $\{a^n\}$ is weakly equicontinuous.  There exists $M>1$ such that $a_j^n(t) \le M$ for any $t \in [0, T]$ and for all $j$ and $n$.  Then 
\begin{align*}
\left| a_j^n(t) - a_j^n(s) \right| & \le \left| \int_s^t \left(  \la^\frac{5}{2}_{j-1} (a_{j-1}^n)^2(\tau) - \la^\frac{5}{2}_j a_j^n(\tau) a_{j+1}^n (\tau)\right) \dtau \right| \\
& \le \left( \la^\frac{5}{2}_{j-1} M^2 +   \la^\frac{5}{2}_j M^2 \right) \left| t-s \right|.
\end{align*} 
Thus 
$$
\dW \left( a^n(t), a^n(s) \right) = \sum_{j = 0} ^\infty 
\frac{1}{\la^{(j)^2}} \frac{\left| a_j^n(t) - a_j^n(s) \right|}{ 1+ \left| a_j^n(t) - a_j^n(s) \right|}
\le c  \left| t-s \right|,
$$
for some constant $c$ independent of $n$.  Then $\{a^n\}$ is an equicontinuous sequence in $C([0, T]; H_W)$ with bounded initial data.  The Arzel\`{a}-Ascoli theorem then implies that $\{a^n\}$ is relatively compact in $C([0, T]; H_W)$.  Passage to a subsequence yields a weakly continuous $H$-valued function $a(t)$ such that 
$$ a^{n_m} \rightarrow a \hspace{5mm} \mathrm{as} \hspace{5mm} n_m \rightarrow \infty  \hspace{5mm}  \mathrm{in}  \hspace{5mm} C([0, T]; H_W).
$$
In particular, $a_j^{n_m} \rightarrow a_j(t)$ as $n_m \rightarrow \infty$ for all $j$ and for all $t \in [0,T]$. Thus $a(0) = a^0$.

Furthermore 
$$
a_j^{n_m} (t) = a_j^{n_m}(0) + \int_0^t \left( \la^\frac{5}{2}_{j-1}(a_{j-1}^{n_m})^2(\tau) -  \la^\frac{5}{2}_j  a_j^{n_m}(\tau) a_{j+1}^{n_m}(\tau)    \right) \dtau,
$$
for $j \le n_m - 1$. Now let $n_m \rightarrow \infty$.  Then 
$$
a_j(t) = a_j(0) + \int_0^t \left(  \la^\frac{5}{2}_{j-1}a_{j-1}^2(\tau) - \la^\frac{5}{2}_j  a_j(\tau) a_{j+1}(\tau)    \right) \dtau.
$$ 
Since $a_j(t)$ is continuous, then $a_j \in C^1 ([0,T])$ and it satisfies our inviscid dyadic system.   
\end{proof}

\begin{lem} \label{lem:scaling}
If $a(t)$ solves \eqref{eq:aj} with initial condition $a(t_0)= a^0$, then $\tilde{a}(t) = \eta a(\eta t)$ is a solution to \eqref{eq:aj} with initial condition $\tilde{a}(t_0) = \eta a(\eta t_0) = \tilde{a}^0$.
\end{lem}
\begin{proof} For $j=0$, the result is trivial.  For $j = 1, 2, 3, ...$, we have
\[
\ddt a_j(t) =  \la_{j-1}^\frac{5}{2} a_{j-1}^2(t) - \la_j^\frac{5}{2}a_j(t)a_{j+1}(t) .
\]
So \begin{align*}
 \ddt \tilde{a}_j(t) &= \eta^2 \ddt a_j(\eta t) \\
&= \eta^2 \left( \la_{j-1}^\frac{5}{2} a_{j-1}^2(\eta t) - \la_j^\frac{5}{2} a_j(\eta t) a_{j+1}(\eta t) \right) \\
& = \la_{j-1}^\frac{5}{2} \left( \eta a_{j-1}(\eta t) \right)^2 -  \la_j^\frac{5}{2} \left( \eta a_j(\eta t)  \right) \left( \eta a_{j+1}(\eta t) \right) \\
& =  \la_{j-1}^\frac{5}{2} \tilde{a}_{j-1}^2(t) - \la_j^\frac{5}{2} \tilde{a}_j(t) \tilde{a}_{j+1}(t).
\end{align*}
Thus $\tilde{a}(t)$ satisfies \eqref{eq:aj} with initial condition $\tilde{a}(t_0) = \eta a(\eta t_0) = \tilde{a}^0$. 
\end{proof}

%%%%%%%%%%%%%%%%%%%%%%%%%%%%%%%%%%%%%%%%%%%%%%
\section{Regularity}

In this section, we study the regularity of positive solutions to the inviscid dyadic model. Apply the change of variables $c_j(t) = \la^{2\theta - \frac{5}{2}}\la^\theta_j a_j(t) $ to rewrite the equations as
\begin{align} \label{eq:cj}
\begin{split}
&\ddt c_j (t) = \la_j^{\frac{5}{2}-\theta} \left( c_{j-1}(t)^2 -\g c_j(t) c_{j+1}(t) \right), \hspace{1cm} j = 1, 2, ...,\\
&c_0(t) = 0,
\end{split}
\end{align}
where $\g = \la^{\frac{5}{2} - 3\theta}$.  We choose 
\begin{align} \label{eq:theta}
\theta = \frac{3}{5}.
\end{align}

\begin{thm} \label{thm:cj}
Let $a(t)$ be a positive solution to \eqref{eq:aj}. There exists $\delta > 0$ such that if $c_j(0) \le \delta <1$ for any $j \in \N$, then $c_j(t) < 1$ for any $j \in \N$ and for all $t>0$.
\end{thm}
\begin{proof}
By the uniqueness proved in \cite{MR3057168} and by Theorem \ref{thm:galerkin}, we have 
\[
a_j(t) = \lim_{m \rightarrow \infty}  a^m_j(t),
\]
where $a_j^m(t)$ is the $m^{th}$ order Galerkin approximation of $a_j(t)$. So it suffices to prove the theorem for the Galerkin approximation $a^m(t)$. We will suppress the notation by omitting the index $m$. 

Fix $m$ and consider the $m^{th}$ Galerkin approximation 
$$ 
c(t) = (c_0(t), c_1(t), ..., c_m(t), 0, ...).
$$ 
Suppose for contradiction there exists $j_0 \in \N$ such that there is a time $T_0 >0$ for which $c_{j_0}(T_0) = 1$ but $c_{j_0}(t) < 1$ for $0<t<T_0$.  Define the set of indices 
$$
I : = \{ j \in \N : j \le m\}.
$$
If $c_j(t) < 1$ for all $j \in I$ for any time $0<t<T_0$, then let $n = j_0$.  Otherwise, if there is a $j \in I$ such that $c_j(t)=1$ for some time $0<t \le T_0$, then define $t_j >0$ to be the time such that $c_j(t_j) = 1$ but $c_j(t) < 1$ for $0 < t< t_j \le T_0$.  If $c_j(t) < 1$ on $(0, T_0]$, then let $t_j = \infty$.  Next define
$$
t^*: = \min_{j \in I} t_j.
$$
Now we can define $n:= \min\{j \in I : c_j(t^*) = 1\}$.  

Note that $n \neq 1$ since  
\begin{align*}
&\ddt c_1(t) =  - \la^{\frac{5}{2}-\theta} \g c_1(t)c_2(t) < 0, \\
& c_1(0) \le \delta <1,
\end{align*}
as $c_1(t), c_2(t) \geq 0$ for all $t>0$. Then $c_1(t)$ is a non-increasing function with initial value strictly below 1. Thus $c_1(t$) cannot cross 1 and hence $t_1 = \infty$.

Now we have a fixed $n \in I$ such that $c_n(t) < 1$ for $0<t<t^*$, $c_n(t^*) = 1$, and $c_j(t) < 1 $ for all other $j \in I \setminus \{n\}$ and $0 < t < t^*$.  We rescale time as $b_j(t) = c_j(\la_n^{\theta - \frac{5}{2}} t)$. 

Note that $b_j(t)$ satisfies 
\begin{equation} \label{eq:bj}
\ddt b_j(t) =   \la_{n-j}^{\theta - \frac{5}{2}}(b_{j-1}(t)^2 - \g b_{j}(t) b_{j+1}(t)),
\end{equation}
where $b_{j}(0) < \delta$ and $b_{j}(t) < 1$ for all $j$ and $0<t<T^*$, where $T^* = \la_n^{\frac{5}{2} - \theta}t^*$. \\

\noindent \textit{Step 1:} A very rough estimate for $b_n(t)$ for $t<t_0$.\\
Fix $k = 0.96$. Note  $\delta < k <1 $. There exists time $t_0 > 0$ such that $k < b_n(t) < 1$ for $t_0 < t < T^*$ and $b_n(t_0) = k$. Recall our assumption on the initial data: $b_n(0) \leq \delta $. So 
\[ 
\ddt b_n(t)  = b_{n-1}(t)^2 - \gamma b_n(t) b_{n+1}(t) < 1,
\]
since $b_{n-1}(t) < 1$ and $b_n(t), b_{n+1}(t) > 0$ for $0<t< T^*$. Thus we have a lower bound on $t_0$: $t_0 \geq k - \delta$. Apply Gronwall's inequality backward in time for $b_n(t)$ to arrive at the following lower bound:
\[
b_n(t)  \geq k - t_0 + t ~~~\mathrm{for} ~~~ t \in [t_0 - k + \delta, t_0] \subseteq [\max\{0, t_0-k\}, t_0].
\] 
See Figures 1 (A) and (B). \\
  
%%%%%%%%%%%%%%%%%%%%%%%%%%%%%%%%  
  
%\begin{figure}
%\caption{}
%\includegraphics[scale = 0.45]{KZFig6}
%\subcaption*{The linear function $f(t) = k-t_0+t$ bounds $b_n(t)$ from below for $t<t_0$.}
%\end{figure}

\begin{figure}
  \centering
    \begin{subfigure}[b]{0.48\textwidth}
    \includegraphics[width=\textwidth]{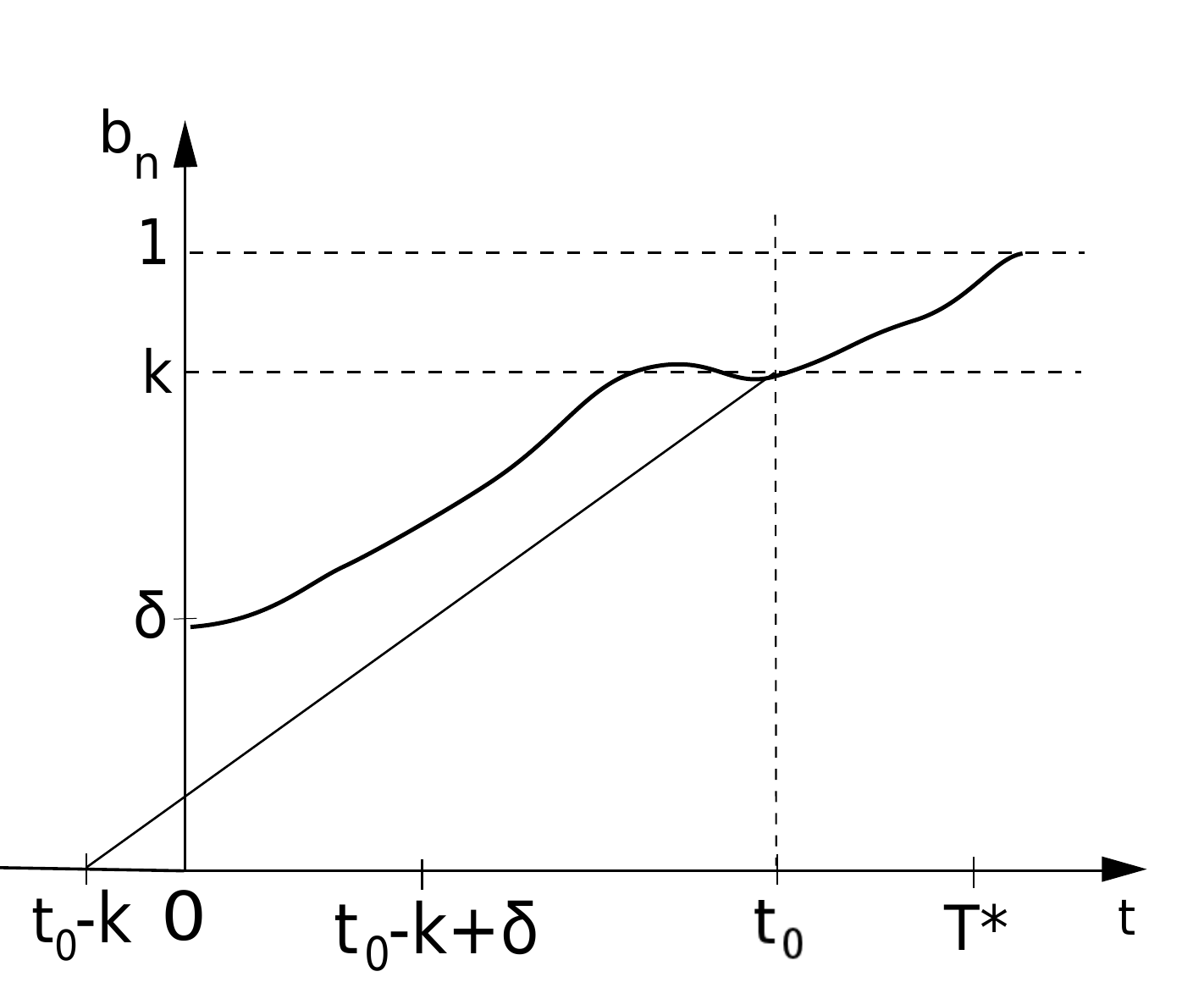}
    \caption{}
    \end{subfigure}%
        ~ %add desired spacing between images, e. g. ~, \quad, \qquad etc.
    \begin{subfigure}[b]{0.48\textwidth}
    \includegraphics[width=\textwidth]{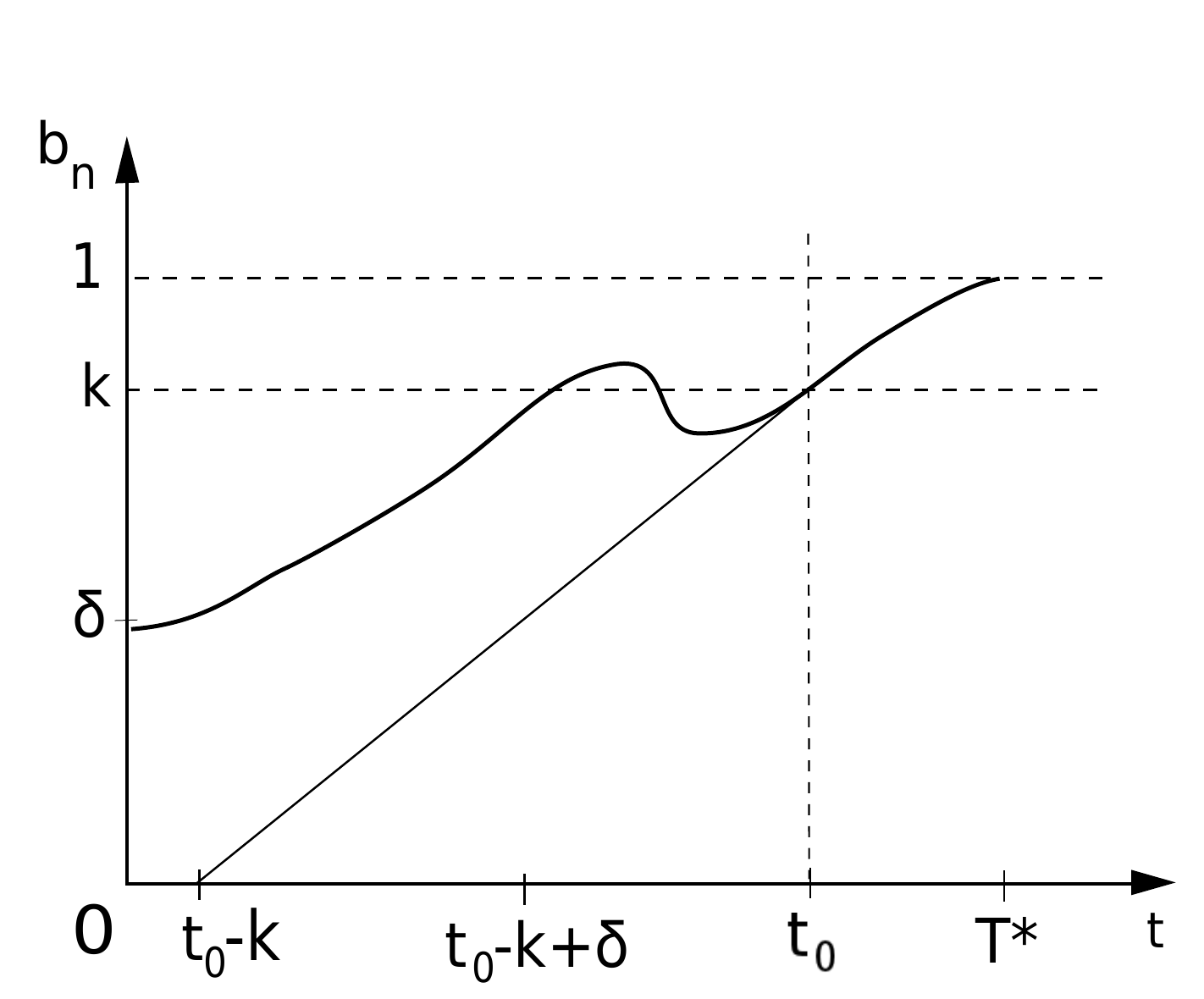}
    \caption{}
    \end{subfigure}
\caption{$k - t_0 +t$ bounds $b_n(t)$ from below for $t \in [t_0 - k + \delta, t_0] \subseteq [\max\{0, t_0-k\}, t_0]$.}
\end{figure}

%%%%%%%%%%%%%%%%%%%%%%%%%%%%%%%%%%%%%%%%%%

\noindent \textit{Step 2:} Estimate $b_{n+1}(t_0)$.\\
For $t \in [t_0 - k + \delta, t_0]$, we have 
\begin{align*}
\ddt b_{n+1}(t) &=   \la^{\frac{5}{2} - \theta} \left(b_{n}(t)^2 - \g b_{n+1}(t) b_{n+2}(t) \right)\\
& \geq \la^{\frac{5}{2} - \theta} \left( (k - t_0 + t)^2 - \g b_{n+1}(t) \right)  \\
 &=  \la^{\frac{5}{2} - \theta} (k-t_0 +t)^2 - \la^{\frac{5}{2} - \theta} \g b_{n+1} (t).
\end{align*}
This yields the initial value problem:
\begin{align*}
&\ddt b_{n+1}(t) + \la^{\frac{5}{2} - \theta} \g b_{n+1}(t) \geq \la^{\frac{5}{2} - \theta} (k- t_0 + t)^2,    \\
&b_{n+1}(t_0 - k + \delta ) \geq 0.
\end{align*}
Apply Gronwall's inequality to find
\begin{align*}
b_{n+1}(t) &\geq b_{n+1}(t_0 - k + \delta) \e^{-\int _{t_0- k + \delta}^t \la^{\frac{5}{2}-\theta}\g \dtau}\\ 
&\hspace{3cm}+ \int_{t_0 - k + \delta}^t \e^{-\int _s^t \la^{\frac{5}{2}-\theta}\g \dtau} \la^{\frac{5}{2} - \theta} (k - t_0 + s)^2 ~\ds \\
&\geq \int_{t_0 - k + \delta}^t \e^{-\la^{\frac{5}{2}-\theta}\g(t-s)} \la^{\frac{5}{2} - \theta} (k - t_0 + s)^2 ~\ds. 
\end{align*}
An application of integration by parts yields
\begin{align*}
b_{n+1}(t)& \ge \la^{\frac{5}{2} - \theta} \e^{-\la^{\frac{5}{2} - \theta} \g t} 
\left[ \left. \frac{(k-t_0+s)^2}{\la^{\frac{5}{2} - \theta} \g}  
\e^{\la^{\frac{5}{2} - \theta} \g s} \right|_{t_0 -k+\delta}^t \right. \\
&\hspace{3cm}
 \left. - \int_{t_0-k+\delta}^t \frac{2(k-t_0+s)}{\la^{\frac{5}{2} - \theta} \g}
\e^{\la^{\frac{5}{2} - \theta} \g s } \ds \right]. \\
\end{align*}
We integrate by parts again to get
\begin{align*}
b_{n+1}(t)&\ge \frac{1}{\g} \e^{-\la^{\frac{5}{2} - \theta}\g t} 
\left[ \left. \left( (k-t_0 +s)^2 \e^{\la^{\frac{5}{2} - \theta}\g s} 
-\frac{2(k-t_0 +s)}{\la^{\frac{5}{2} - \theta} \g} \e^{\la^{\frac{5}{2}- \theta} \g s} \right)\right|_{t_0 -k + \delta}^t  \right. \\
& \hspace{7cm} \left. + \int_{t_0 -k +\delta}^t \frac{2}{\la^{\frac{5}{2} - \theta} \g} \e^{\la^{\frac{5}{2} - \theta} \g s} \ds \right] \\
&= \frac{1}{\g} \e^{-\la^{\frac{5}{2} - \theta}\g t} 
\left[ (k-t_0 +s)^2 \e^{\la^{\frac{5}{2} - \theta}\g s} 
-\frac{2(k-t_0 +s)}{\la^{\frac{5}{2} - \theta} \g} \e^{\la^{\frac{5}{2}- \theta} \g s}  \right. \\
&\hspace{6cm} + \left. \left. \frac{2}{\la^{5-2\theta} \g^2} \e^{\la^{\frac{5}{2}- \theta} \g s} \right] \right|_{t_0 -k + \delta}^t \\
&=  \frac{(k-t_0 + t)^2}{ \g} - 
\frac{2(k-t_0 + t)}{\la^{\frac{5}{2} - \theta} \g^2} + 
\frac{2}{\la^{5 - 2\theta}\g^3}  \\
& \hspace{3cm}- \e^{-\la^{\frac{5}{2}-\theta}\g (t - t_0 + k - \delta)} 
 \left[ \frac{\delta^2}{\g} - 
\frac{2\delta}{\la^{\frac{5}{2} - \theta}  \g^2} + 
\frac{2}{\la^ {5 - 2\theta} \g^3} \right].
 \end{align*}

\noindent Thus we have  
\begin{multline*}
b_{n+1}(t_0) \geq \left[ \frac{k^2}{\g} - 
\frac{2k}{\la^{\frac{5}{2} - \theta} \g^2} + 
\frac{2}{\la^{5 - 2\theta} \g ^3} \right] \\
 -\e^{-\la^{\frac{5}{2}-\theta} \g (k-\delta)} 
 \left[ \frac{\delta^2}{ \g} - 
\frac{2\delta}{\la^{\frac{5}{2} - \theta}  \g^2}+ \frac{2}{\la^ {5 - 2\theta} \g^3} \right] =:B(\delta).
\end{multline*}
As $\delta$ tends to 0,
\[
B(\delta) \rightarrow \left[ \frac{k^2}{\g} - \frac{2k}{\la^{\frac{5}{2} - \theta}\g^2} + 
\frac{2}{\la^{5 - 2\theta} \g ^3} \right] - \frac{2 \e^{-\la^{\frac{5}{2}-\theta} \g k}}{\la^ {5 - 2\theta} \g^3} > 0.447.
\]
So there exists $\delta$ small enough that $B(\delta) \ge 0.447 :=B$, which we will use as the bound on initial condition $b_{n+1}(t_0)$. \\

\noindent \textit{Step 3:} Estimate $b_{n\pm1}(t)$ for  $t_0< t \le T^*$.\\
By our assumptions when $t>t_0$, in particular that $b_{n -2} (t) \leq 1$ and $b_n(t) \geq k$, we get the following inequality from equation \eqref{eq:bj}:
\begin{align*}
&\ddt b_{n-1}(t) \leq \la^{\theta - \frac{5}{2}} \left( 1 - k\g b_{n-1}(t) \right),  \\
&b_{n-1}(t_0) \leq 1.
\end{align*}
Then by Gronwall's inequality,
\begin{align*}
b_{n-1}(t) &\leq b_{n-1}(t_0) \e^{-\int_{t_0}^t \la^{\theta - \frac{5}{2}} k \g \dtau} + \int_{t_0}^t \la^{\theta - \frac{5}{2}} \e^ {-\int_s^t \la^{\theta - \frac{5}{2}} k \g \dtau}
~\mathrm{d}s\\
&\leq \e^{-\la^{\theta - \frac{5}{2}} k \g (t-t_0)} + \int_{t_0}^t \la^{\theta - \frac{5}{2}} \e^ {- \la^{\theta - \frac{5}{2}} k \g (t-s)} ~\mathrm{d}s\\
& = \e^{-\la^{\theta - \frac{5}{2}} k \g (t-t_0)} \left( 1 - \frac{1}{k \g}   \right) + \frac{1}{k\g} =: \hat{b}_{n-1}(t).
\end{align*}
By our assumptions when $t_0< t \le T^*$, in particular that $b_{n +2}(t) \leq 1$ and $b_n(t) \geq k$, we get the following inequality from equation \eqref{eq:bj}:
\begin{align*}
&\ddt b_{n+1}(t) \geq \la^{\frac{5}{2} - \theta} \left(k^2 - \g b_{n+1}(t)\right), \\
&b_{n+1} (t_0) \geq B .
\end{align*}
Then by Gronwall's inequality,
\begin{align*}
b_{n+1}(t) &\geq b_{n+1}(t_0) \e^{-\int_{t_0}^t \la^{\frac{5}{2} - \theta} \g \dtau} + \int_{t_0}^t \e^{-\int_s^t \la^{\frac{5}{2} - \theta} \g \dtau} \la^{\frac{5}{2} - \theta} k^2 ~\ds \\ 
&\geq B \e^{- \la^{\frac{5}{2} - \theta} \g (t-t_0)} + \int_{t_0}^t \e^{- \la^{\frac{5}{2} - \theta} \g (t-s)} \la^{\frac{5}{2} - \theta} k^2 ~\ds \\ 
&=  B \e^{- \la^{\frac{5}{2} - \theta} \g (t-t_0)} + \left. \frac{k^2}{\g}\e^{- \la^{\frac{5}{2} - \theta} \g t} \e^{ \la^{\frac{5}{2} - \theta} \g s} \right|_{t_0}^t \\
&= \e^{-\la^{\frac{5}{2} - \theta} \g (t-t_0)} \left( B -  \frac{k^2}{\g} \right) +  \frac{k^2}{\g} =: \tilde{b}_{n+1}(t).
\end{align*}

\noindent \textit{Step 4:} Estimate $b_n(t)$ for $t_0< t \le T^*$.\\
We use the bounds on $b_{n\pm1}(t)$ from above to find an upperbound on $b_n(t)$:
\begin{align*}
&\ddt b_n(t) = b_{n-1}^2(t) - \g b_n(t) b_{n+1}(t) \le \hat{b}_{n-1}^2(t) - \g b_n(t) \tilde{b}_{n+1}(t),   \\
&b_n (t_0) \le k.  
\end{align*}
Another application of Gronwall's inequality yields
\begin{align*}
b_n(t) &\le k \e^{-\int_{t_0}^t \g \tilde{b}_{n+1} \dtau} + \int_{t_0}^t \e^{-\int_s^t \g \hat{b}_{n+1} \dtau} \hat{b}_{n-1}^2 \ds \\
&= k \e^{-\int_{t_0}^t \g \e^{-\la^{\frac{5}{2} - \theta} \g (\tau-t_0)} \left( B -  \frac{k^2}{\g} \right) +  k^2 \dtau} \\
&+ \int_{t_0}^t \e^{-\int_s^t \g \e^{-\la^{\frac{5}{2} - \theta} \g (\tau-t_0)} \left( B -  \frac{k^2}{\g} \right) +  k^2 \dtau} \left( \e^{-\la^{\theta - \frac{5}{2}} k \g (s-t_0)} \left( 1 - \frac{1}{k \g}   \right) + \frac{1}{k\g} \right)^2 \ds \\
&= k \e^{ 
 \la^{\theta - \frac{5}{2}}  \left( B -  \frac{k^2}{\g} \right) \left( \e^{-\la^{\frac{5}{2} - \theta} \g (t-t_0)} -1 \right) - k^2(t-t_0)     
 } \\
 &+ \int_{t_0}^t  \e^{  \la^{\theta - \frac{5}{2}} \left( B -  \frac{k^2}{\g} \right) \left( \e^{-\la^{\frac{5}{2} - \theta} \g (t-t_0)} -\e^{-\la^{\frac{5}{2} - \theta} \g (s-t_0)} \right)  - k^2(t-s) } \\
 & \hspace{4cm} \cdot
 \left( \e^{-\la^{\theta - \frac{5}{2}} k \g (s-t_0)} \left( 1 - \frac{1}{k \g}   \right) + \frac{1}{k\g} \right)^2 \ds \\
 &=: \beta (t).
 %&= k \e^{ 
 %\la^{\theta - \frac{5}{2}}  \left( B -  \frac{k^2}{\g} \right) \left( \e^{-\la^{\frac{5}{2} - \theta} \g (t-t_0)} -1 \right) - k^2(t-t_0) }  + \e^{  \la^{\theta - \frac{5}{2}} \left( B -  \frac{k^2}{\g} \right) \e^{-\la^{\frac{5}{2} - \theta} \g (t-t_0)} - k^2(t-t_0) } \\
 %& \cdot \int_{t_0}^t \e^{  -\la^{\theta - \frac{5}{2}} \left( B -  \frac{k^2}{\g} \right) \e^{-\la^{\frac{5}{2} - \theta} \g (s-t_0)} + k^2(s-t_0) }
% \left( \e^{-\la^{\theta - \frac{5}{2}} k \g (s-t_0)} \left( 1 - \frac{1}{k \g}   \right) + \frac{1}{k\g} \right)^2 \ds. \\
\end{align*}
%A change of variables yields
%\begin{align*}
%b_n(t)\le &k \e^{ 
% \la^{\theta - \frac{5}{2}}  \left( B -  \frac{k^2}{\g} \right) \left( \e^{-\la^{\frac{5}{2} - \theta} \g (t-t_0)} -1 \right) - k^2(t-t_0) } + \e^{  \la^{\theta - \frac{5}{2}} \left( B -  \frac{k^2}{\g} \right) \e^{-\la^{\frac{5}{2} - \theta} \g (t-t_0)} - k^2(t-t_0) } \\
% & \cdot \int_0^{t-t_0} \e^{  -\la^{\theta - \frac{5}{2}} \left( B -  \frac{k^2}{\g} \right) \e^{-\la^{\frac{5}{2} - \theta} \g \sigma} + k^2\sigma }
% \left( \e^{-\la^{\theta - \frac{5}{2}} k \g \sigma} \left( 1 - \frac{1}{k \g}   \right) + \frac{1}{k\g} \right)^2 \dsig. \\
%\end{align*}
Then 
\begin{align*}
\ddt \beta(t) &=  k \e^{ \la^{\theta - \frac{5}{2}}  \left( B -  \frac{k^2}{\g} \right) \left( \e^{-\la^{\frac{5}{2} - \theta} \g (t-t_0)} -1 \right) - k^2(t-t_0)} 
\left( \left( k^2 - B \g \right) \e^{-\la^{\frac{5}{2} - \theta} \g (t-t_0)} -k^2 \right)  \\
 & + \left( \e^{-\la^{\theta - \frac{5}{2}} k \g (t-t_0)} \left( 1 - \frac{1}{k \g}   \right) + \frac{1}{k\g} \right)^2 \\ 
 &+ \left( \left( k^2 - B \g \right) \e^{-\la^{\frac{5}{2} - \theta} \g (t-t_0)} -k^2 \right)\int_{t_0}^t  \e^{  \la^{\theta - \frac{5}{2}} \left( B -  \frac{k^2}{\g} \right) \left( \e^{-\la^{\frac{5}{2} - \theta} \g (t-t_0)} -\e^{-\la^{\frac{5}{2} - \theta} \g (s-t_0)} \right)  - k^2(t-s) } \\
 & \hspace{6cm} \cdot  \left( \e^{-\la^{\theta - \frac{5}{2}} k \g (s-t_0)} \left( 1 - \frac{1}{k \g}   \right) + \frac{1}{k\g} \right)^2 \ds. \\
 \end{align*}
The exponent $\la^{\theta - \frac{5}{2}} \left( B -  \frac{k^2}{\g} \right)\left( \e^{-\la^{\frac{5}{2} - \theta} \g (t-t_0)} -\e^{-\la^{\frac{5}{2} - \theta} \g (s-t_0)} \right)$ is nonnegative, thus 
\begin{align*}
\ddt \beta(t) & \le \left( \e^{-\la^{\theta - \frac{5}{2}} k \g (t-t_0)} \left( 1 - \frac{1}{k \g}   \right) + \frac{1}{k\g} \right)^2 \\
&+ \left( \left( k^2 - B \g \right) \e^{-\la^{\frac{5}{2} - \theta} \g (t-t_0)} -k^2 \right) \int_{t_0}^t  \frac{1}{(k \g)^2} \e^{-k^2(t-s)} \ds \\
&= \left( \e^{-\la^{\theta - \frac{5}{2}} k \g (t-t_0)} \left( 1 - \frac{1}{k \g}   \right) + \frac{1}{k\g} \right)^2 \\
&+ \left( \left( 1-  \frac{B \g}{k^2} \right) \e^{-\la^{\frac{5}{2} - \theta} \g (t-t_0)} -1 \right) \frac{1}{(k \g)^2} \left( 1 - \e^{-k^2(t-t_0)} \right) \\
& = \e^{-2\la^{\theta - \frac{5}{2}} k \g (t-t_0)} \left( 1 - \frac{1}{k \g}   \right)^2 + \frac{2}{k\g} \left( 1 - \frac{1}{k \g}   \right)  \e^{-\la^{\theta - \frac{5}{2}} k \g (t-t_0)}  + \frac{1}{(k\g)^2} \e^{-k^2(t-t_0)}  \\
&+ \frac{1}{(k\g)^2}  \left( 1-  \frac{B \g}{k^2} \right) \e^{-\la^{\frac{5}{2} - \theta} \g (t-t_0)} -\frac{1}{(k\g)^2} \left( 1-  \frac{B \g}{k^2} \right) \e^{-(\la^{\frac{5}{2} - \theta} \g +k^2) (t-t_0)}. \\
\end{align*}
We have shown exponential decay for the derivative $\beta'(t)$ and thus it suffices to show $\beta(t) <1$ on a finite interval, which can be accomplished easily numerically since $\beta(t)$ is given explicitly. Hence $b_n(t) < 1$ for all $t>0$, which contradicts our assumption that $b_n(t)$ is the first $b_j$ that crosses 1. Thus $b_j(t) < 1$ for any $j \in \N$ and for all $t>0$. The conclusion extends to $c_j(t)$.
\end{proof}

This leads to our main result:
\begin{thm} \label{thm:aj}
Let $a(t)$ be a positive solution to \eqref{eq:aj} such that 
\[
\sup_j  \la_j^\theta a_j (0) = M 
\]
for some $M< \infty$, then
\[
\sup_j  \la_j^\theta a_j (t) < \frac{M}{\delta} 
\]
for all $t> 0$.
\end{thm}
\begin{proof}
By the above Lemma \ref{lem:scaling}, we have that if $a_j(t)$ solves  \eqref{eq:aj} with $a_j(0) = a_j^0$, then $\tilde{a}_j(t) = \eta a_j(\eta t)$ is a solution to \eqref{eq:aj} with initial condition $\tilde{a}_j(0) = \eta a_j(0) = \eta a_j^0$. In particular, this is true for $\eta = \frac{\delta}{M}$.  Since $\sup_j  \la_j^\theta a_j(0) = M$, we have 
\[
\sup_j \la_j^\theta  \tilde{a}_j(0) = \eta M = \frac{\delta}{M} M = \delta.
\]
Define
\[
b_j(t) : = \la_j^\theta  \tilde{a}_j (t) .
\]
Then we have 
\[
\sup_j b_j(0) < \delta.
\]
Given such an upper bound on the initial condition of $b_j(t)$, then recall that the theorem above yields 
\[ 
b_j(t) < 1  
\]
for all $t>0$.  Then
\[
\sup_j b_j(t) =  \la_j^\theta  \tilde{a}_j (t) = \la_j^\theta \eta a_j (t)  < 1 .
\]
Therefore
\[
 \sup_j  \la_j^\theta a_j(t) < \frac{1}{\eta} = \frac{M}{\delta} 
\]
for all $t>0$.
\end{proof}

Similar to Theorem 10 in \cite{MR3057168}, we obtain the following 
\begin{cor} \label{cor:decay}
There exists a constant $k(\theta) > 0$ such that 
\[
\sup_j  \la_j^\theta a_j(t) < k(\theta) |a(0)|^\frac{2}{3} t^{-\frac{1}{3}} , \hspace{5mm} \mathrm{for~all} ~~t>0
\]
for every positive solution $a(t)$ of \eqref{eq:aj} with $a(0)$ in $H$.
\end{cor}
\begin{proof}
By Theorem \ref{thm:aj}, 
\[
\sup_j  \la_j^\theta a_j (t) < \frac{1}{\delta}  \sup_j \la_j^\theta a_j (s)
\]
for all $s \in [0, t]$. By \cite{MR3057168}, there exists a constant $f(\theta)>0$ such that 
\[
\mathscr{L}\left\{ t > 0 : a_j (t) > l_j ~\mathrm{for~some}~j  \right\} \le  f(\theta) |a(0)|^2 \sum_{j = 1}^\infty \frac{1}{\la_j^\frac{5}{2} l_j^3}, 
\]
where $\mathscr{L}$ denotes the Lebesgue measure and $(l_j)_{j\ge 1}$ is any positive, non-increasing sequence. 
Let 
\[
l_j = \frac{ f(\theta)^\frac{1}{3} M^\frac{1}{3} |a(0)| ^\frac{2}{3} }{\la_j^\theta t^\frac{1}{3}} ,
\]
where $M$ is such that
\[
\sum_{j=1}^\infty \la_j^{3\theta - \frac{5}{2}} < M.
\]
This series converges since $3\theta - \frac{5}{2} < 0$ by \eqref{eq:theta}.
Then we get 
\begin{align*}
\mathscr{L} \left\{ s> 0 :  a_j(s) > l_j  ~\mathrm{for~some}~j  \right\} &\le f(\theta) |a(0)|^2 \sum_{j=1}^\infty  \frac{1}{\la_j^\frac{5}{2}} \frac{\la_j^{3\theta} t}{f(\theta) |a(0)|^2 M} \\
& = \frac{1}{M} \sum_{j=1}^\infty \la_j^{3\theta - \frac{5}{2}}t \\
&< t.
\end{align*}
Thus for some $s\in [0, t]$, we have $a_j(s) \le l_j$ for all $j$.  Thus
\[
\la_j^\theta a_j(s) \le f(\theta)^\frac{1}{3} M^\frac{1}{3} |a(0)|^\frac{2}{3} t^{-\frac{1}{3}} 
\]
for all $j$. Then
\begin{align*}
\sup_j \la_j^\theta a_j(s) &\le f(\theta)^\frac{1}{3} M^\frac{1}{3} |a(0)|^\frac{2}{3} t^{-\frac{1}{3}} \\
& = k(\theta) |a(0)|^\frac{2}{3} t^{-\frac{1}{3}}, ~~~~~~k~\mathrm{constant~depending~on~}~\theta ,
\end{align*}
which yields the result
\[
\sup_j \la_j^\theta a_j(t) \le \frac{1}{\delta}k(\theta) |a(0)|^\frac{2}{3} t^{-\frac{1}{3}}
\]
for all $t>0$, as desired.
\end{proof}

%\nocite{*}
\bibliography{Dyadic_Template}{}
\bibliographystyle{plain}

\end{document}